\documentclass[letter,english]{article}

\usepackage[utf8]{inputenc}
\usepackage[T1]{fontenc}
\usepackage{babel}
\usepackage[babel]{csquotes}

\usepackage{amsmath}
\usepackage{amsfonts}
\usepackage{amsthm}
\usepackage{bbm}
\usepackage{amssymb}
\usepackage{IEEEtrantools}
\usepackage{url}

\newtheorem{theorem}{Theorem}[section]
\newtheorem{proposition}[theorem]{Proposition}
\newtheorem{lemma}[theorem]{Lemma}

\theoremstyle{definition}
\newtheorem{definition}[theorem]{Definition}

\newtheorem{remark}{Remark}

\DeclareMathOperator{\Ad}{Ad}

\DeclareMathOperator{\End}{End}

\DeclareMathOperator{\rk}{rk}
\DeclareMathOperator{\card}{card}
\DeclareMathOperator{\Span}{Span}
\DeclareMathOperator{\Hom}{Hom}

\newcommand{\C}{\mathbb{C}}

\newcommand{\sk}{\mathsf{k}}
\newcommand{\bG}{\mathbf{G}}
\newcommand{\T}{\mathcal{T}}

\title{A spectral gap theorem in simple Lie groups}
\author{Yves Benoist and Nicolas de Saxcé
\thanks{N.S. is supported by ERC AdG Grant 267259}
}

\begin{document}

\maketitle

\begin{abstract}
We establish the spectral gap property for dense subgroups generated by algebraic elements in any compact simple Lie group, generalizing earlier results of Bourgain and Gamburd for unitary groups.
\end{abstract}

\section{Introduction}

The purpose of the paper is to study the spectral gap property for measures on a compact simple Lie group $G$. If $\mu$ is a Borel probability measure on $G$, we say that $\mu$ has a \emph{spectral gap} if the spectral radius of the corresponding operator on $L^2_0(G)$ -- the space of mean-zero square integrable functions on $G$ -- is strictly less than $1$.
We also say that $\mu$ is \emph{almost Diophantine} if it satisfies, for some positive constants $C_1$ and $c_2$, for $n$ large enough and for any proper closed subgroup $H$,
$$\mu^{*n}(\{x\in G \,|\, d(x,H)\leq e^{-C_1n}\}) \leq e^{-c_2n}.$$
Using the discretized Product Theorem proved in \cite{saxceproducttheorem} and the techniques developped by Bourgain and Gamburd in \cite{bourgaingamburdsu2} for the group $SU(2)$, we prove the following theorem.

\begin{theorem}\label{adsgi}
Let $G$ be a connected compact simple Lie group and $\mu$ be a Borel probability measure on $G$. Then $\mu$ has a spectral gap if and only if it is almost Diophantine.
\end{theorem}

A measure $\mu$ on the compact simple Lie group $G$ is called \emph{adapted} if its support generates a dense subgroup of $G$. It is not known whether every adapted probability measure on the compact simple Lie group $G$ is almost Diophantine, but it is natural to conjecture a affirmative answer to this question. In this direction, Bourgain and Gamburd proved that if $\mu$ is an adapted probability measure on $SU(d)$ supported on elements with algebraic entries, then $\mu$ has a spectral gap. We generalize their result to an arbitrary simple group, and prove the following, using the theory of random matrix products over arbitrary local fields, as exposed in \cite{benoistquintwalk}.

\begin{theorem}\label{algebraici}
Let $G$ be a connected compact simple Lie group and $\mathcal{U}$ a fixed basis for its Lie algebra. Let $\mu$ be an adapted probability measure on $G$ and assume that for any $g$ in the support of $\mu$, the matrix of $\Ad g$ in the basis $\mathcal{U}$ has algebraic entries. Then $\mu$ is almost Diophantine, and therefore has a spectral gap.
\end{theorem} 

In the case $G$ is the group $SO(n)$ of rotations of the Euclidean space of dimension $n$, Theorem~\ref{algebraici} is used by Lindenstrauss and Varjú \cite{lindenstraussvarju} to study absolute continuity of self-similar measures defined by isometries of the Euclidean space described by matrices with algebraic coefficients.

\paragraph{}

The plan of the paper is simple: in Section~\ref{spectralgap}, we prove Theorem~\ref{adsgi}, in Section~\ref{algebraic}, we prove Theorem~\ref{algebraici}.

\paragraph{}

For us, a compact simple Lie group will be a compact real Lie group whose Lie algebra is simple. We will also make use of some classical notation:
\begin{itemize}
\item[-] The Landau notation: $O(\epsilon)$ stands for a quantity bounded in absolute value by $C\epsilon$, for some constant $C$ (generally depending on the ambient group $G$).
\item[-] The Vinogradov notation: we write $x\ll y$ if, $x\leq Cy$ for some constant $C$ (again, possibly depending on the ambient group). We will also write $x\simeq y$ if $x\ll y$ and $x\gg y$, and similarly. For two real valued functions $\varphi$ and $\psi$ on $G$, we write $\varphi\ll\psi$ if there exists an absolute constant $C$ such that for all $x$ in $G$, $\varphi(x)\leq C\cdot\psi(x)$.
\end{itemize}

\paragraph{Acknowledgements.} The authors are grateful to the Israel Institute for Advanced Studies, where this work was done, during the 2013 Arithmetic and Dynamics semester.

\section{The spectral gap property}\label{spectralgap}

Let $G$ be a connected compact simple Lie group. If $\mu$ is a Borel probability measure on $G$, we define an averaging operator $T_\mu$ on the space $L^2_0(G)$ of mean-zero square-integrable functions by the formula
$$T_\mu f (x) = \int_G f(xg)\,d\mu(g),\quad \forall f\in L^2_0(G).$$

\begin{definition}
We say that a probability measure $\mu$ on $G$ has a \emph{spectral gap} if the spectral radius of the averaging operator $T_\mu$ on the space $L^2_0(G)$ is strictly less than one.
\end{definition}

The purpose of this section is to relate the spectral gap property to the following Diophantine property of measures.

\begin{definition}
We say that a probability measure $\mu$ on $G$ is \emph{almost Diophantine} if there exist positive constants $C_1$ and $c_2$ such that for $n$ large enough, for any proper closed connected subgroup $H$,
\begin{equation}\label{wd}
\mu^{*n}(H^{(e^{-C_1n})}) \leq e^{-c_2n}.
\end{equation}
where $H^{(\rho)}$ denotes the neighborhood of size $\rho$ of the closed subgroup $H$: $H^{(\rho)}=\{x\in G \,|\, d(x,H)\leq\rho\}$.
\end{definition}

With this definition, we have the following theorem.

\begin{theorem}[Spectral gap for almost Diophantine measures]\label{adsg}
Let $G$ be a connected compact simple Lie group. A Borel probability measure $\mu$ on $G$ has a spectral gap if and only if it is almost Diophantine.
\end{theorem}

\begin{remark}\label{symmetric}
The spectral radius of the averaging operator $T_\mu$ on $L^2_0(G)$ is less than one if and only if the spectral radius of $T_\mu T_{\check{\mu}}=T_{\mu*\check{\mu}}$ is less than one. This shows that it will be enough to prove the Theorem~\ref{adsg} in the case $\mu$ is symmetric. 
\end{remark}

We start by proving the trivial implication: if $\mu$ has a spectral gap, then it must be almost Diophantine.

\begin{proof}[Spectral gap $\Longrightarrow$ Almost Diophantine] 
Suppose $\mu$ has a spectral gap, and let $c>0$ such that the spectral radius of $T_\mu$ satisfies $RS(T_\mu)\leq e^{-c}$. Let $d$ be the dimension of $G$ and let $H$ be a maximal proper closed subgroup of $G$ of dimension $p$. For $\delta>0$, we can bound the $L^2$-norm of the indicator function of the $2\delta$-neighborhood of $H$:
$$\|\mathbbm{1}_{H^{(2\delta)}}\|_2 \ll \delta^{\frac{d-p}{2}}.$$
Therefore, for $n$ larger than $\frac{d-p}{2c}\log\frac{1}{\delta}$, we have
$$\|T_\mu^n \mathbbm{1}_{H^{(2\delta)}}\|_2 \ll \delta^{d-p}.$$ 
Making the left-hand side explicit, we find
$$\sqrt{\int_G \mu^{*n}(xH^{(2\delta)})^2\,dx} \ll \delta^{d-p}$$
and this implies,
$$\mu^{*n}(H^{(\delta)}) \ll \delta^{\frac{d-p}{2}}.$$
Choosing $C_1\leq\frac{2c}{d-p}$ and $c_2=c$, and letting $\delta=e^{-C_1n}$, this shows that $\mu$ is almost Diophantine.
\end{proof}

To prove the converse implication in Theorem~\ref{adsg}, we use the strategy developped by Bourgain and Gamburd. If $A$ is a subset of a metric space, for $\delta>0$, we denote by $N(A,\delta)$ the minimal cardinality of a covering of $A$ by balls of radius $\delta$. We have the following Product Theorem \cite[Theorem~3.9]{saxceproducttheorem}.

\begin{theorem}\label{sigmakappa}
Let $G$ be a simple Lie group of dimension $d$. There exists a neighborhood $U$ of the identity in $G$ such that the following holds.\\
Given $\alpha\in (0,d)$ and $\kappa>0$, there exists $\epsilon_0=\epsilon_0(\alpha,\kappa)>0$ and $\tau=\tau(\alpha,\kappa)>0$ such that, for $\delta>0$ sufficiently small, if $A\subset U$ is a set satisfying
\begin{enumerate}
\item $N(A,\delta) \leq \delta^{-d+\alpha-\epsilon_0}$,
\item for all $\rho\geq\delta$, $N(A,\rho)\geq \rho^{-\kappa}\delta^{\epsilon_0}$,
\item $N(AAA,\delta)\leq\delta^{-\epsilon_0}N(A,\delta)$,
\end{enumerate}
then $A$ is included in a neighborhood of size $\delta^\tau$ of a proper closed connected subgroup of $G$.
\end{theorem}

We will use Theorem~\ref{sigmakappa} to derive a flattening statement for measures. 
For $\delta>0$, we let
$$P_\delta=\frac{\mathbbm{1}_{B(1,\delta)}}{|B(1,\delta)|},$$
(where $|\cdot|$ is the volume associated to the Haar probability measure on $G$)
and if $\mu$ is a probability measure on $G$, we denote by $\mu_\delta$ the function approximating $\mu$ at scale $\delta$:
$$\mu_\delta = \mu*P_\delta.$$

\begin{lemma}[$L^2$-flattening]\label{flattening}
Let $G$ be a connected compact simple Lie group.
Given $\alpha,\kappa>0$, there exists $\epsilon>0$ such that the following holds for any $\delta>0$ small enough.\\
Suppose $\mu$ is a symmetric Borel probability measure on $G$ such that one has
\begin{enumerate}
\item $\|\mu_\delta\|_2^2 \geq \delta^{-\alpha}$,
\item for any $\rho\geq\delta$ and any closed connected subgroup $H$, $\mu*\mu(H^{(\rho)}) \leq \delta^{-\epsilon}\rho^\kappa$.
\end{enumerate}
Then,
$$\|\mu_\delta*\mu_\delta\|_2 \leq \delta^\epsilon \|\mu_\delta\|_2.$$
\end{lemma}

The proof goes by approximating the measure $\mu_\delta$ by dyadic level sets. We say that a collection of sets $\{X_i\}_{i\in I}$ is \emph{essentially disjoint} if for some constant $C$ depending only on the ambient group $G$, any intersection of more than $C$ distinct sets $X_i$ is empty. We will use the following lemma.

\begin{lemma}\label{dyadic}
Let $G$ be a compact Lie group, $\mu$ a Borel probability measure on $G$ and $\delta>0$. There exist subsets $A_i$, $0\leq i \ll \log\frac{1}{\delta}$ such that
\begin{enumerate}
\item $\mu_\delta \ll \sum_i 2^i \mathbbm{1}_{A_i} \ll \mu_{4\delta}$
\item Each $A_i$ is an essentially disjoint union of balls of radius $\delta$.
\end{enumerate}
\end{lemma}
\begin{proof}
A proof in the case $G=SU(2)$ is given in \cite{lindenstrausssaxcesu2} and also applies in this more general setting, mutatis mutandis.
\end{proof}

To derive Lemma~\ref{flattening}, we will also use the non-commutative Balog-Szemerédi-Gowers Lemma, due to Tao. If $A$ and $B$ are two subsets of a metric group $G$, we define the \emph{multiplicative energy} of $A$ and $B$ at scale $\delta$ by
$$E_\delta(A,B) = N(\{(a,b,a',b')\in A\times B\times A\times B \,|\, d(ab,a'b')\leq\delta\},\delta).$$
(See~\cite{taoestimates} for elementary properties.) We have the following important theorem (see Tao~\cite[Theorem~6.10]{taoestimates}).

\begin{theorem}[Non-commutative Balog-Szemerédi-Gowers Lemma]\label{bsgt}
Let $G$ be a compact Lie group with a Riemannian metric. There exists a constant $C>0$ depending only on $G$ such that the following holds for any $\delta>0$ and any $K\geq 2$.\\
Suppose that $A$ and $B$ are non-empty subsets of $G$ such that
$$E_\delta(A,B) \geq \frac1K N(A,\delta)^{\frac32}N(B,\delta)^{\frac32}.$$
Then there exists a $K^C$-approximate subgroup $H$ and elements $x,y$ in $G$ such that
\begin{itemize}
\item $N(H,\delta) \leq K^C\cdot N(A,\delta)^{\frac12}N(B,\delta)^{\frac12}$
\item $N(A\cap xH,\delta) \geq K^{-C}\cdot N(A,\delta)$
\item $N(B\cap Hy,\delta) \geq K^{-C}\cdot N(B,\delta)$.
\end{itemize}
\end{theorem}

Recall that a subset $H$ of $G$ is called a \emph{$K$-approximate subgroup} if it is symmetric and there exists a finite symmetric set $X\subset H^2$ of cardinality at most $K$ such that $HH\subset XH$.
We are now ready to prove Lemma~\ref{flattening}.

\begin{proof}[Proof of Lemma~\ref{flattening}]
Write
$$\mu_\delta \ll \sum_i 2^i \mathbbm{1}_{A_i} \ll \mu_{4\delta}$$
as in Lemma~\ref{dyadic}. 
Note that for all $i$, one has
$$2^i|A_i|^{\frac12} = \|2^i\mathbbm{1}_{A_i}\|_2 \ll \|\mu_{4\delta}\|_2 \simeq \|\mu_\delta\|_2,$$
and
$$2^i|A_i| \simeq 2^i \delta^d N(A_i,\delta) \ll 1.$$
Assume for a contradiction that for some $\epsilon>0$,
$$\|\mu_\delta*\mu_\delta\|_2 \geq \delta^\epsilon \|\mu_\delta\|_2,$$
with $\delta>0$ arbitrarily small. This gives,
\begin{align*}
\delta^\epsilon \|\mu_\delta\|_2 & \ll \|\sum_{i,j} 2^i \mathbbm{1}_{A_i}*2^j\mathbbm{1}_{A_j}\|_2\\
& \leq \sum_{i,j} \|2^i \mathbbm{1}_{A_i}*2^j\mathbbm{1}_{A_j}\|_2,
\end{align*}
and as the sum on the right-hand side contains at most $O((\log \delta)^2)$ terms, we must have, for some $i$ and $j$,
$$\|2^i\mathbbm{1}_{A_i}*2^j\mathbbm{1}_{A_j}\|_2 \gg \frac{\delta^\epsilon}{(\log\delta)^2}\|\mu_\delta\|_2 \geq \delta^{O(\epsilon)}\|\mu_\delta\|_2.$$
Therefore,
\begin{equation}\label{seq}
\delta^{O(\epsilon)}\|\mu_\delta\|_2 \leq \|2^i\mathbbm{1}_{A_i}*2^j\mathbbm{1}_{A_j}\|_2
 \leq \|2^i\mathbbm{1}_{A_i}\|_1 \|2^j\mathbbm{1}_{A_j}\|_2
 \ll 2^i|A_i| \|\mu_{\delta}\|_2.
\end{equation}
This implies,
\begin{equation}\label{seqq}
2^i |A_i| = \delta^{O(\epsilon)}\quad \mbox{and similarly}\quad 2^j|A_j|=\delta^{O(\epsilon)}.
\end{equation}
So we have the following lower bound on the multiplicative energy of $A_i$ and $A_j$:
\begin{align*}
E_\delta(A_i,A_j) & \gg \delta^{-3d}\|\mathbbm{1}_{A_i}*\mathbbm{1}_{A_j}\|_2^2\\
& \geq \delta^{-3d+O(\epsilon)}2^{-2i-2j}\|\mu_\delta\|_2^2\\
& \geq \delta^{-3d+O(\epsilon)} 2^{-i-j}|A_i|^{\frac{1}{2}}|A_j|^{\frac{1}{2}} = \delta^{O(\epsilon)} N(A_i,\delta)^{\frac{3}{2}}N(A_j,\delta)^{\frac{3}{2}}.
\end{align*}
By Theorem~\ref{bsgt}, there exists a $\delta^{-O(\epsilon)}$-approximate subgroup $\tilde{H}$ and elements $x,y$ in $G$ such that
\begin{equation}\label{smallh}
N(\tilde{H},\delta) \leq \delta^{-O(\epsilon)} N(A_i,\delta)^{\frac12} N(A_j,\delta)^{\frac12},
\end{equation}
\begin{equation}\label{hclosetoa}
N(x\tilde{H}\cap A_i,\delta) \geq \delta^{O(\epsilon)} N(A_i,\delta) \quad \mbox{and}\quad N(\tilde{H}y\cap A_j,\delta) \geq \delta^{O(\epsilon)} N(A_j,\delta).
\end{equation}
We may replace $\tilde{H}$ by its $\delta$-neighborhood, and then, $\mu_\delta(x\tilde{H})\geq\delta^{O(\epsilon)}$. Let $U$ be a neighborhood of the identity in $G$ as in Theorem~\ref{sigmakappa}, let $r>0$ be such that $B(1,2r)\subset U$, and cover $x\tilde{H}$ by $O(1)$ balls of radius $r$. One of these balls $B$ must satisfy $\mu_\delta(x\tilde{H}\cap B)\geq\delta^{O(\epsilon)}$ and thus,
$$\mu_\delta*\mu_\delta(\tilde{H}^2\cap U) \geq \mu_\delta(\tilde{H}x^{-1}\cap B^{-1})\mu_\delta(x\tilde{H}\cap B) \geq \delta^{O(\epsilon)}.$$
On the other hand, by (\ref{seq}) and (\ref{seqq}),
$$\delta^{O(\epsilon)}\|\mu_\delta\|_2 \leq \|2^i\mathbbm{1}_{A_i}\|_1\|2^j\mathbbm{1}_{A_j}\|_2 \ll \|2^j\mathbbm{1}_{A_j}\|_2 \leq \delta^{-O(\epsilon)}2^{j/2},$$
so that $2^j\geq\delta^{-\alpha+O(\epsilon)}$ and similarly $2^i\geq\delta^{-\alpha+O(\epsilon)}$. This implies
$$N(A_j,\delta)\leq\delta^{-d+\alpha-O(\epsilon)} \quad\mbox{and similarly}\quad N(A_i,\delta)\leq\delta^{-d+\alpha-O(\epsilon)}.$$
The set $\tilde{H}$ is a $\delta^{-O(\epsilon)}$-approximate subgroup, so $N(\tilde{H}^2,\delta)\leq\delta^{-O(\epsilon)}N(\tilde{H},\delta)$. Recalling Inequality (\ref{smallh}), we find
$$N(\tilde{H}^2\cap U,\delta) \leq N(\tilde{H}^2,\delta) \leq \delta^{-d+\alpha-O(\epsilon)}.$$
On the other hand, $\mu_\delta*\mu_\delta(\tilde{H}^2\cap U)\geq \delta^{O(\epsilon)}$ so the second assumption on $\mu_\delta$ forces, for any $\rho\geq\delta$ (note that any ball of radius $\rho$ is included in the $\rho$-neighborhood of some proper closed connected subgroup),
$$N(\tilde{H}^2\cap U,\rho) \geq \rho^{-\kappa}\delta^{O(\epsilon)}.$$
Thus, provided we have chosen $\epsilon>0$ small enough, the set $\tilde{H}^2\cap U$ satisfies the assumptions of Theorem~\ref{sigmakappa}, and so must be included in the $\delta^{\tau}$-neighborhood of a proper closed connected subgroup $H$ of $G$, contradicting the assumption $\mu*\mu(H^{(\delta^\tau)})\leq \delta^{-\epsilon}\delta^{\kappa\tau}$.
\end{proof}

The idea is now to apply repeatedly that Flattening Lemma to obtain:

\begin{lemma}\label{little}
Let $\mu$ be a symmetric almost Diophantine measure on a connected compact simple Lie group $G$.
There exists a constant $C_0=C_0(\mu)$ such that for any $\delta=e^{-C_0n}>0$ small enough,
$$\|(\mu^{*C_0\log\frac{1}{\delta}})_\delta\|_2 \leq \delta^{-\frac{1}{4}}.$$
\end{lemma}

\begin{remark}
The constant $\frac{1}{4}$ could be replaced in this lemma by any fixed positive constant $\alpha$. Of course, $C_0$ would then depend on $\alpha$. 
\end{remark}

\begin{proof}
We first check that a suitable power $\nu=\mu^{c\log\frac{1}{\delta}}$ satisfies the second condition of Lemma~\ref{flattening}. Since $\mu$ is almost Diophantine, taking $n=\frac{1}{C_1}\log\frac{1}{\delta}$ in Equation~(\ref{wd}) shows that when $\delta<\delta_0$, for any proper closed connected subgroup $H$,
$$\mu^{*\frac{1}{C_1}\log\frac{1}{\delta}}(H^{(\delta)}) \leq \delta^{\frac{c_2}{C_1}}.$$
If $xH$ is a left coset of a closed subgroup $H$ and $m$ any symmetric measure, we have
$$m(xH^{(\delta)})^2 \leq m*m(H^{(2\delta)}).$$
Therefore, denoting $c=\frac{1}{4C_1}$ and $\kappa=\frac{c_2}{3C_1}$, we have, for all $\delta<\delta_0$, for any left coset $xH$ of a proper closed connected subgroup,
$$\mu^{*2c\log\frac{1}{\delta}}(xH^{(\delta)}) \leq \delta^{\kappa}.$$
Now, if $H$ is a closed subgroup and $m$ and $m'$ are any two probability measures on $G$, we have
$$m*m'(H^{(\delta)}) \leq \sup_{x\in G} m'(xH^{(\delta)}).$$
Therefore, if $\delta<\rho<\delta_0$, we have, for any proper closed connected subgroup $H$,
$$\mu^{*2c\log\frac{1}{\delta}}(H^{(\rho)}) \leq \max_x \mu^{*2c\log\frac{1}{\rho}}(xH^{(\rho)}) \leq \rho^\kappa.$$
In other terms, for $\delta>0$ small enough, the measure $\nu:=\mu^{*c\log\frac{1}{\delta}}$ satisfies the second condition of Lemma~\ref{flattening}.\\
We now apply Lemma~\ref{flattening} repeatedly, starting with the measure $\nu$. If $\|\nu_\delta\|_2\leq \delta^{-\frac{1}{4}}$, then we have what we want. Otherwise, Lemma~\ref{flattening} applied to $\nu_\delta$ with $\alpha=\frac{1}{2}$ shows that
$$\|(\nu*\nu)_\delta\|_2 \ll \|\nu_\delta*\nu_\delta\|_2 \leq \delta^\epsilon \|\nu_\delta\|_2.$$
We then repeat the same procedure, replacing $\nu$ by $\nu*\nu$, and so on (note that the computations made above for $\nu$ also show that all the convolution powers of $\nu$ will satisfy the second condition of Lemma~\ref{flattening}). After at most $\frac{d}{\epsilon}$ iterations, the procedure must stop, i.e. we must have,
$$\|(\mu^{*C_0\log\frac{1}{\delta}})_\delta\|_2 = \|(\nu^{*2^{\frac{d}{\epsilon}}})_\delta\|_2 \leq \delta^{-\frac{1}{4}}.$$
\end{proof}

The end of the proof of Theorem~\ref{adsg} relies on the high-multiplicity of irreducible representations in the regular representation $L^2(G)$. Recall that the irreducible representations of $G$ are in bijection with dominant analytically integral weights (see e.g. \cite{knapplg}). We denote by $\pi_\lambda$ the irreducible representation of $G$ with highest weight $\lambda$. If $\mu$ is a finite Borel measure on $G$, the Fourier coefficient of $\mu$ at $\lambda$ is
$$\hat{\mu}(\lambda) = \int_G \pi_\lambda(g)\,d\mu(g).$$
By Lemma~\ref{little}, all we need to show is the following.

\begin{lemma}
Let $\mu$ be a Borel probability measure on a compact semisimple Lie group $G$ such that for some constant $C$, for all $\delta=e^{-Cn}>0$ small enough ($n$ a positive integer),
$$\|(\mu^{*C\log\frac{1}{\delta}})_\delta\|_2 \leq \delta^{-\frac{1}{4}}.$$
Then $\mu$ has a spectral gap in $L^2(G)$.
\end{lemma}
\begin{proof}
Since the representation $V_\lambda$ occurs in $L^2(G)$ with multiplicity $\dim V_\lambda$, the Parseval Formula for $(\mu^{*C\log\frac{1}{\delta}})_\delta$ gives
\begin{equation}\label{parseval}
\|(\mu^{*C\log\frac{1}{\delta}})_\delta\|_2^2  = \sum_\lambda (\dim V_\lambda)\|\hat{\mu}(\lambda)^{C\log\frac{1}{\delta}}\hat{P_\delta}(\lambda)\|_{HS}^2,
\end{equation}
where $\|\cdot\|_{HS}$ is the Hilbert-Schmidt norm.
Moreover, it is easily seen that we may bound the distance (in operator norm) from $\hat{P_\delta}(\lambda)$ to the identity (see for instance \cite[Lemme~3.1]{saxcetroudimensionnel}): for some constant $c>0$ depending only on $G$, we have, whenever $\|\lambda\|\leq c\delta^{-1}$,
$$\|\hat{P_\delta}(\lambda)-Id_{V_\lambda}\|_{op} \leq \frac{1}{2}.$$
Therefore for any $\lambda$ such that $\|\lambda\|\leq c\delta^{-1}$, using (\ref{parseval}) and the assumption of the lemma,
\begin{equation}\label{ten}
\delta^{-\frac{1}{2}} \geq \frac{1}{4} (\dim V_\lambda)\|\hat{\mu}(\lambda)^{C\log\frac{1}{\delta}}\|_{op}^2.
\end{equation}
Now, as a consequence of the Weyl dimension Formula, we have, for some constant $c$ depending only on $G$, for any representation $V_\lambda$ with highest weight $\lambda$ \cite[Lemme~3.2]{saxcetroudimensionnel},
$$\dim V_\lambda \geq c\|\lambda\|.$$
Taking $\lambda$ with $e^{-C}c\delta^{-1}\leq\|\lambda\|\leq c\delta^{-1}$ in the above equation~(\ref{ten}), we find
$$\|\hat{\mu}(\lambda)^{C\log\frac{1}{\delta}}\|_{op}^2 \ll \delta^{\frac{1}{2}}.$$
However, the spectral radius of an operator $T$ satisfies, for any integer,
$$RS(T) \leq \|T^n\|_{op}^{\frac{1}{n}},$$
so that for some absolute constant $K$, we have
\begin{align*}
RS(\hat{\mu}(\lambda)) 
& \leq (K\delta^{\frac{1}{4}})^{\frac{1}{C\log\frac{1}{\delta}}}\\
& = e^{-\frac{1}{4C}} K^{\frac{1}{C\log\frac{1}{\delta}}}
\end{align*}
which is bounded away from $1$ as long as $\delta$ is sufficiently small, i.e. as long as $\lambda$ is sufficiently large. As the spectral radius of $T_\mu$ in $L^2_0(G)$ is equal to the supremum of all $RS(\hat{\mu}(\lambda))$ for $\lambda\neq 0$, this finishes the proof.
\end{proof}

\section{Measures supported on algebraic elements}\label{algebraic}

In this section, we fix a basis for the Lie algebra $\mathfrak{g}$. We say that an element $g\in G$ is \emph{algebraic} if the entries of the matrix of $\Ad g$ in that fixed basis are algebraic numbers. Recall that a probability measure on $G$ is called \emph{adapted} if its support generates a dense subgroup of $G$. We want to prove the following.

\begin{theorem}\label{algebraicspectralgap}
Let $G$ be a connected compact simple Lie group. If $\mu$ is an adapted probability measure on $G$ whose support consists of algebraic elements, then $\mu$ has a spectral gap.
\end{theorem}

\begin{remark} We have already explained in Remark~\ref{symmetric} that it is enough to prove such a theorem for a symmetric measure $\mu$. Moreover, if $\mu$ is symmetric, under the assumptions of the theorem, we may always find a symmetric finitely supported adapted measure $\nu$ that is absolutely continuous with respect to $\mu$. It is readily seen that if $\nu$ has a spectral gap, then so has $\mu$, so we may assume in the proof of Theorem~\ref{algebraicspectralgap} that $\mu$ is finitely supported.
\end{remark}

The proof has two parts. First, we show that, given a proper closed subgroup $H$, the probability $\mu^{*n}(H)$ decays exponentially, with a rate that does not depend on $H$. This part is based on the theory of product of random matrices, as developed by Furstenberg, Guivarc'h and others; the central input is Theorem~\ref{decay} below. The difficult point in the proof is to reduce to the case where the subgroup generated by the support of $\mu$ acts proximally. While writing this paper, we learnt from Emmanuel Breuillard that an alternative approach was to derived an improved version of Theorem~\ref{decay} that applies also to some non-proximal representations \cite{breuillardsuperstrong}. Some partial results on this issue were also obtained previously by Aoun \cite{aountransience}. 

Then, we show that when the support of $\mu$ consists of algebraic elements, the measure $\mu$ is almost Diophantine. This second part is based on an application of the effective arithmetic Nullstellensatz, and relies crucially on the algebraic assumption on the elements of the support of $\mu$.

\subsection{Transience of closed subgroups}

We want to prove the following.

\begin{proposition}\label{subgroups}
Let $\mu$ be an adapted finitely supported symmetric probability measure on a connected compact simple Lie group $G$. Then, there exists a constant $\kappa=\kappa(\mu)$ such that for $n\geq n_0$, for any proper closed subgroup $H<G$,
$$\mu^{*n}(H) \leq e^{-\kappa n}.$$
\end{proposition}

The proposition is based on the following lemma.

\begin{lemma}\label{goodrepresentation}
Let $\Gamma=\langle S\rangle$ be a finitely generated dense subgroup in $G$. There exists a finite collection of vector spaces $\mathcal{S}_i$, $1\leq i\leq s$, over local fields $K_i$, such that the following holds:
\begin{itemize}
\item for each $i\in\{1,\dots,s\}$, the group $\Gamma$ acts proximally and strongly irreducibly on $\mathcal{S}_i$;
\item for any proper closed subgroup $H<G$ such that $\Gamma\cap H$ is infinite, there exists an $i\in\{1,\dots,s\}$ for which $\Gamma\cap H$ stabilizes a proper linear subspace of $\mathcal{S}_i$.
\end{itemize}
\end{lemma}

Let us explain how this lemma implies Proposition~\ref{subgroups}, when combined with the following important result of random matrix products theory \cite[Proposition~12.3]{benoistquintwalk} (see also \cite[Theorem~4.4]{bflm}).

\begin{theorem}\label{decay}
Let $K$ be a local field and $\mathcal{S}$ be a finite dimensional vector space over $K$. Suppose $\mu$ is a measure on $GL(\mathcal{S})$ such that the semigroup $\Gamma$ generated by the support of $\mu$ acts proximally on $\mathcal{S}$.
Then, there exists a constant $\kappa=\kappa(\mu)$ such that for any integer $n$ large enough, for any vector $v\in\mathcal{S}$ and any hyperplane $V<\mathcal{S}$,
$$\mu^{*n}(\{g\in GL(\mathcal{S}) \,|\, g\cdot v \in V\}) \leq e^{-\kappa n}.$$
\end{theorem}

\begin{proof}[Proof of Proposition~\ref{subgroups}]
Let $\Gamma$ be the group generated by the support of $\mu$.
Given a proper closed connected subgroup $H$ of $G$, we distinguish two cases.\\
\underline{First case:} $\Gamma\cap H$ is finite.\\
By Selberg's Lemma, $\Gamma$ contains a torsion free subgroup of finite index $N_0$. Hence the cardinality of $\Gamma\cap H$ is bounded by $N_0$ and the uniform exponential decay of $\mu^{*n}(H)=\mu^{*n}(\Gamma\cap H)$ is a direct consequence of Kesten's Theorem \cite[Corollary~3]{kesten} since $\Gamma$ is not amenable.\\
\underline{Second case:} $\Gamma\cap H$ is infinite.\\
Let $\mathcal{S}_i$, $1\leq i\leq s$, be the vector spaces given by Lemma~\ref{goodrepresentation}. For each $i$, the measure $\mu$ may be viewed as a measure on $GL(\mathcal{S}_i)$. Choose $\kappa>0$ such that the conclusion of Theorem~\ref{decay} holds for each $\mathcal{S}_i$.\\
Choose $i$ such that $\Gamma\cap H$ stabilizes a proper subspace $L$ of $\mathcal{S}_i$. We then have, for $n$ large enough,
$$\mu^{*n}(\{g\in\Gamma \,|\, g\cdot L=L\}) \leq e^{-\kappa n},$$
so that
$$\mu^{*n}(H) = \mu^{*n}(H\cap\Gamma) \leq e^{-\kappa n}.$$
\end{proof}

Before turning to the proof of Lemma~\ref{goodrepresentation}, let us recall the setting.
The group $\Gamma$ is a dense finitely generated free subgroup of the connected compact simple group $G$, and $\sk$ is the field generated by the coefficients of the elements $\Ad g$, for $g$ in $\Gamma$.
As $\Gamma$ is dense in $G$, we may view $G$ as the group of real points of an algebraic group $\bG$ defined over $\sk$. Whenever $K$ is a field containing $\sk$, we will denote by $\bG(K)$ the group of $K$-points of $\bG$. Similarly, if $V$ is a linear representation of $\bG$ defined over $K$, we will write $V(K)$ for the associated $K$-vector space, on which $\bG(K)$ acts.\\
In the case when $\Gamma$ acts proximally on the adjoint representation $\mathfrak{g}(K)$, for some local field $K$ containing $\sk$, the proof of Lemma~\ref{goodrepresentation} is substantially simpler. This is the content of the next lemma.

\begin{lemma}\label{adjoint}
Assume that $\Gamma$ acts proximally on $\mathfrak{g}(K)$, for some local field $K$ containing $\sk$. Then,
\begin{itemize}
\item the group $\Gamma$ acts proximally and strongly irreducibly on $\mathfrak{g}(K)$;
\item for any proper closed subgroup $H<G$ such that $\Gamma\cap H$ is infinite, $\Gamma\cap H$ stabilizes a proper linear subspace of $\mathfrak{g}(K)$.
\end{itemize}
\end{lemma}
\begin{proof}
By assumption, $\Gamma$ acts proximally on $\mathfrak{g}(K)$.
As $\Gamma$ is dense in $G$, it is Zariski dense in $\bG(K)$, and therefore $\Gamma$ acts strongly irreducibly on $\mathfrak{g}(K)$.\\
Now if $H$ is a proper closed infinite subgroup of $G$ such that $\Gamma\cap H$ is infinite, then $\Gamma\cap H$ stabilizes the (complex) Lie algebra of the Zariski closure of $\Gamma\cap H$. This is a proper subspace $L<\mathfrak{g}_{\C}$ defined over $\sk$ (and hence, over $K$), so that $\Gamma\cap H$ stabilizes a proper subspace of $\mathfrak{g}(K)$.
\end{proof}

Let $\Delta\subset E$ ($E$ a Euclidean space of dimension $\rk G$) be the root system of $G$, choose a basis $\Pi$ for $\Delta$, and let $C$ be the associated Weyl chamber. If $\omega$ is a dominant weight, with associated irreducible representation $V^\omega$, we denote by $\omega^*$ the dominant weight of the dual irreducible representation $(V^\omega)^*$. We observe the following:

\begin{lemma}\label{tilde}
Let $\tilde{\alpha}$ be the largest root of $\Delta$. Either $\tilde{\alpha}=\omega$ is a fundamental weight, or $\tilde{\alpha}=\omega+\omega^*$ is the sum of a fundamental weight and its dual (those two might coincide).
\end{lemma}
\begin{proof}
Let $\rho$ be the sum of all fundamental weights of $\Delta$. Choose a fundamental weight $\omega$ minimizing $\langle\omega,\rho\rangle$. The adjoint representation can be viewed as a subrepresentation of $\End V^\omega\simeq V^\omega\otimes (V^\omega)^*$. Comparing the highest weights, we find that $\tilde{\alpha}$ can be written
$$\tilde{\alpha} = \omega+\omega^* -\sum_i n_i\alpha_i, \quad n_i\in\mathbb{N},\ \alpha_i\ \mbox{simple roots}.$$
Taking the inner product with $\rho$, we find that $\langle\tilde{\alpha},\rho\rangle \leq 2\langle\omega,\rho\rangle$ and in case of equality, we must have all $n_i$ equal to zero i.e. $\tilde{\alpha}=\omega+\omega^*$. On the other hand, if the inequality is strict, by minimality of $\langle\omega,\rho\rangle$, the dominant weight $\tilde{\alpha}$ must be fundamental (not necessarily $\omega$, though). This proves the lemma.
\end{proof}

Finally, we recall the following fact.

\begin{lemma}\label{opposition}
Assume $\Gamma$ acts proximally on $V^{\omega}(K)$, for some local field $K$ containing $\sk$. Then, $\Gamma$ acts proximally on $V^{\omega+\omega^*}(K)$. 
\end{lemma}
\begin{proof}
This is an immediate consequence of the fact that if $\Gamma$ acts proximally on a vector space $V$, then we may find an element $\gamma$ in $\Gamma$ such that both $\gamma$ and $\gamma^{-1}$ act proximally on $V$, see \cite[Lemme~3.9]{benoistluminy}.
\end{proof}

According to Lemma~\ref{tilde}, write $\tilde{\alpha}=\omega$ or $\tilde{\alpha}=\omega+\omega^*$. Putting together Lemma~\ref{adjoint} and Lemma~\ref{opposition}, we find that Lemma~\ref{goodrepresentation} holds whenever $\Gamma$ acts proximally on $V^\omega(K)$ (or $V^{\omega^*}(K)$) for some local field $K$. Therefore, for the rest of the proof of Lemma~\ref{goodrepresentation}, we assume (writing the largest root $\tilde{\alpha}=\omega+\omega^*$ or $\tilde{\alpha}=\omega$, for some fundamental weight $\omega$):

\begin{equation}\label{notomega}
\mbox{There is no local field $K$ such that $\Gamma$ acts proximally on $V^\omega(K)$.}
\end{equation}

To prove Lemma~\ref{goodrepresentation}, we start by defining a certain family of irreducible complex representations of $G$.
For any nonzero vector $X$ in the Weyl chamber $C$ of $\Delta$, we let 
$$\mathcal{E}_X = \{\alpha\in\Delta \,|\, \langle\alpha, X\rangle\ \mbox{is maximal}\}$$
and
$$m_X=\card \mathcal{E}_X.$$
Note that the largest root $\tilde{\alpha}$ of $\Delta$ always belongs to $\mathcal{E}_X$ so that $\mathcal{E}_X = \{\alpha\in\Delta \,|\, \langle\tilde{\alpha}-\alpha, X\rangle=0 \}$.\\
Finally, we define a dominant weight $\omega_X$ by
$$\omega_X = \sum_{\alpha\in\mathcal{E}_X} \alpha,$$
and denote by $\mathcal{S}_X$ the irreducible representation of $G$ with highest weight $\omega_X$.\\
A simple way to check that $\omega_X$ is indeed a dominant weight is to construct $\mathcal{S}_X$ explicitly as follows.
Write the decomposition of $\mathfrak{g}_{\mathbb{C}}$ into root spaces for some maximal torus $T$:
$$\mathfrak{g}_{\mathbb{C}}=\mathfrak{t}_{\C}\oplus\left(\bigoplus_{\alpha\in\Delta}\mathfrak{g}_\alpha\right).$$
Each $\mathfrak{g}_\alpha$ is one-dimensional, so write $\mathfrak{g}_\alpha=\mathbb{C}E_\alpha$. The representation $\mathcal{S}_X$ is the subrepresentation of $\bigwedge^{m_X}\mathfrak{g}_{\mathbb{C}}$ generated by the vector 
$$\xi_X=\bigwedge_{\alpha\in\mathcal{E}_X}E_\alpha \in \bigwedge^{m_X}\mathfrak{g}_{\mathbb{C}}.$$

The spaces $\mathcal{S}_i$ of Lemma~\ref{goodrepresentation} will be constructed as representations $\mathcal{S}_X(K)$, where the local field $K$ will be suitably chosen as to arrange that the action of $\Gamma$ is proximal. The difficult point will be to prove the existence of a proper stable subspace under $\Gamma\cap H$, when $H$ is a closed subgroup. For that, one crucial observation is the following fact about faces of root systems.

\begin{lemma}\label{faces}
Let $\Delta$ be an irreducible root system with a given basis $\Pi$. Denote by $\tilde{\alpha}$ the largest root of $\Delta$, and let $X$ be a nonzero vector in the Weyl chamber $C$. In the case $\tilde{\alpha}=\omega+\omega^*$ and $\omega\neq\omega^*$, assume $X$ not collinear to $\omega$ nor to $\omega^*$. We define the \emph{face of $\Delta$ associated to $X$} by
$$\mathcal{E}_X=\{\alpha\in\Delta \,|\, \langle\tilde{\alpha}-\alpha,X\rangle = 0\},$$
and denote by $W_{\tilde{\alpha}}$ the stabilizer of $\tilde{\alpha}$ in the Weyl group $W$ of $\Delta$.
Then,
$$\bigcap_{w\in W_{\tilde{\alpha}}} w\cdot \mathcal{E}_X = \{\tilde{\alpha}\}.$$
\end{lemma}
\begin{proof}
Letting $\mathcal{E}_X'=\tilde{\alpha}-\mathcal{E}_X$, we want to check that
$$\bigcap_{w\in W_{\tilde{\alpha}}} w\cdot \mathcal{E}_X' = \{0\}.$$
For sake of clarity, we deal first with the case when $\tilde{\alpha}$ is proportional to some fundamental weight $\omega=\omega_{i_0}$. Any element $u$ in $\mathcal{E}_X'$ can be written $u=\tilde{\alpha}-\alpha$, so that 
$$\langle u,\tilde{\alpha}\rangle = \|\tilde{\alpha}\|^2-\langle\alpha,\tilde{\alpha}\rangle,$$
and, as $\tilde{\alpha}$ has maximal norm among the roots, this shows,
\begin{equation}\label{positive}
\forall u\in \mathcal{E}_X'\backslash\{0\},\ \langle u,\tilde{\alpha}\rangle > 0.
\end{equation}
On the other hand, since the largest root $\tilde{\alpha}$ is proportional to a fundamental weight, the elements of $E$ invariant under $W_{\tilde{\alpha}}$ are proportional to $\tilde{\alpha}$. This implies that the element $\frac{1}{|W_{\tilde{\alpha}}|}\sum_{w\in W_{\tilde{\alpha}}}w \in \End E$ is just the orthogonal projection to $\mathbb{R}\tilde{\alpha}$, so that
$$\frac{1}{|W_{\tilde{\alpha}}|}\sum_{w\in W_{\tilde{\alpha}}}w\cdot X = \langle X,\frac{\tilde{\alpha}}{\|\tilde{\alpha}\|^2}\rangle\tilde{\alpha},$$
is a nonzero multiple of $\tilde{\alpha}$. This implies in particular that
$$\bigcap_{w\in W_{\tilde{\alpha}}} w\cdot X^\perp \subset \tilde{\alpha}^\perp.$$
Recalling (\ref{positive}), we indeed find
$$\bigcap_{w\in W_{\tilde{\alpha}}} w\cdot \mathcal{E}_X' \subset \mathcal{E}_X'\cap \bigcap_{w\in W_{\tilde{\alpha}}} w\cdot X^\perp \subset \mathcal{E}_X'\cap \tilde{\alpha}^\perp = \{0\}.$$
We deal now with the case $\tilde{\alpha}=\omega+\omega^*$, with $\omega\neq\omega^*$. This means that the group $G$ is of type $A_\ell$, i.e. locally isomorphic to $SU(\ell+1)$. Note that this is exactly the case studied by Bourgain and Gamburd in \cite{bourgaingamburdsud}. We may modify the above argument in the following way. The element  $\frac{1}{|W_{\tilde{\alpha}}|}\sum_{w\in W_{\tilde{\alpha}}}w$ is the orthogonal projection on the subspace $\mathbb{R}\omega\oplus\mathbb{R}\omega^*$. As $X$ is not collinear to $\omega$ nor to $\omega^*$, we have 
$$\frac{1}{|W_{\tilde{\alpha}}|}\sum_{w\in W_{\tilde{\alpha}}}w\cdot X = a\omega+b\omega^*, \quad\mbox{for some}\ a,b>0$$
so that
$$\bigcap_{w\in W_{\tilde{\alpha}}} w\cdot X^\perp \subset (a\omega+b\omega^*)^\perp.$$
Then we observe that any element $u$ in $\mathcal{E}_X'$ is a sum of simple roots:
$$u=\sum_{\alpha\in\Pi} n_\alpha\alpha$$
and as $\tilde{\alpha}=\omega+\omega^*$ has maximal norm among the roots, we must have $n_\alpha\geq 1$ for $\alpha$ the simple root corresponding to $\omega$ or $\omega^*$. This implies in particular
$$\forall u\in \mathcal{E}_X'\backslash\{0\},\ \langle u,a\omega+b\omega^*\rangle > 0.$$
As before, this yields
$$\bigcap_{w\in W_{\tilde{\alpha}}} w\cdot \mathcal{E}_X' \subset \mathcal{E}_X'\cap \bigcap_{w\in W_{\tilde{\alpha}}} w\cdot X^\perp \subset \mathcal{E}_X'\cap (a\omega+b\omega^*)^\perp = \{0\}.$$
\end{proof}

This property of root systems implies the following result about non-irreducibility of the representations $\mathcal{S}_X$ under proper subgroups of $G$.

\begin{lemma}\label{nonirreducible}
Let $G$ be a connected compact simple Lie group with root system $\Delta$, let $X$ be a nonzero vector in the Weyl chamber $C$. In the case $\tilde{\alpha}=\omega+\omega^*$ and $\omega\neq\omega^*$, assume $X$ is not collinear to $\omega$ nor to $\omega^*$. If $H$ is a proper closed positive dimensional subgroup of $G$ such that for some $\gamma$ in $H$, the vector $\xi_X$ above is an eigenvector of $\gamma$ whose associated eigenvalue has multiplicity one. Then, the representation $\mathcal{S}_X$ is not irreducible under the action of $H$.
\end{lemma}
\begin{proof}
Denote by $L$ the complexification of the Lie algebra of $H$, by $L^\perp$ its orthogonal for the Killing form, and write
$$\textstyle{\bigwedge}^{m_X}\mathfrak{g}_{\mathbb{C}} = \bigoplus_{j=0}^{m_X} \textstyle{\bigwedge^jL\wedge\bigwedge^{m_X-j}L^\perp}.$$
All the subspaces on the right-hand side of the formula are stable under the action of $\gamma$ (in fact, of $H$), so that the eigenvector $\xi_X$, whose associated eigenvalue has multiplicity one, must belong to one of them, say
\begin{equation}\label{xiin}
\textstyle\xi_X \in \bigwedge^jL\wedge\bigwedge^{m_X-j}L^\perp.
\end{equation}
The subspace $\mathcal{S}_X\cap \bigwedge^jL\wedge\bigwedge^{m_X-j}L^\perp$ is a nonzero subspace of $\mathcal{S}_X$ that is invariant under $H$. Suppose for a contradiction that it is equal to the whole of $\mathcal{S}_X$, i.e. that
\begin{equation}\label{included}
\textstyle\mathcal{S}_X\subset \bigwedge^jL\wedge\bigwedge^{m_X-j}L^\perp.
\end{equation}
Let $F$ be the subspace of $\mathfrak{g}_{\mathbb{C}}$ generated by the $E_\alpha$, for $\alpha$ in $\mathcal{E}_X$. By (\ref{xiin}), we have
$$F = F\cap L \oplus F\cap L^\perp.$$
As the largest root $\tilde{\alpha}$ is always in $\mathcal{E}_X$, the vector $E_{\tilde{\alpha}}$ is in $F$, and therefore,
$$p_L(E_{\tilde{\alpha}})\in F,$$
where $p_L$ denotes the orthogonal projections from $\mathfrak{g}_{\mathbb{C}}$ to $L$.
Now, let $w$ be an element of the Weyl group of $\Delta$ fixing $\tilde{\alpha}$. By (\ref{included}) and the fact that $\mathcal{S}_X$ is stable under $G$, we have
$$\textstyle w\cdot\xi_X \in \bigwedge^jL\wedge\bigwedge^{m_X-j}L^\perp.$$
Reasoning as before, this yields, since $\tilde{\alpha}$ is invariant under $w$,
$$p_L(E_{\tilde{\alpha}})\in w\cdot F.$$
Therefore, letting $w$ describe the stabilizer $W_{\tilde{\alpha}}$ of the largest root, we obtain
$$p_L(E_{\tilde{\alpha}})\in \bigcap_{w\in W_{\tilde{\alpha}}}w\cdot F.$$
However, by Lemma~\ref{faces}, the intersection on the right reduces to $\mathbb{C}E_{\tilde{\alpha}}$. If $p_L(E_{\tilde{\alpha}})\neq 0$, we find $E_{\tilde{\alpha}} \in L$.
Otherwise, $E_{\tilde{\alpha}} \in L^\perp$.
To conclude, we observe that by (\ref{included}) and the fact that $\mathcal{S}_X$ is stable under $G$, we have, for any $g$ in $G$,
$$g\cdot\xi_X \in \bigwedge^jL\wedge\bigwedge^{m_X-j}L^\perp,$$
so that we can reason exactly as before, just conjugating the maximal torus $T$, the root-spaces and the space $F$ by the element $g$. This yields
$$g\cdot E_{\tilde{\alpha}}\in L \quad\mbox{or}\quad g\cdot E_{\tilde{\alpha}}\in L^\perp.$$
Exchanging if necessary $L$ and $L^\perp$, we may assume without loss of generality that for a set $A\subset G$ of positive Haar measure in $G$, we have
$$\forall g\in A,\ g\cdot E_{\tilde{\alpha}} \in L,$$
which is easily seen to imply $L=\mathfrak{g}_{\mathbb{C}}$ contradicting the assumption that $H$ is a proper closed connected subgroup of $G$.\\
Thus, we have shown that $\mathcal{S}_X\cap \bigwedge^jL\wedge\bigwedge^{m_X-j}L^\perp$ is a proper subspace of $\mathcal{S}_X$ that is invariant under $H$. In particular, $\mathcal{S}_X$ is not irreducible under $H$.
\end{proof}

\begin{remark}
Note that the fact that $\mathcal{S}_X$ is not irreducible under $H$ also implies that it is not irreducible under any conjugate $aHa^{-1}$ of $H$.
\end{remark}

We are now ready to conclude the proof of Proposition~\ref{subgroups} by deriving Lemma~\ref{goodrepresentation}.

\begin{proof}[Proof of Lemma~\ref{goodrepresentation}]
Clearly, it suffices to deal with maximal proper closed subgroups $H$. There are only finitely many such maximal subgroups, up to conjugation by elements of $G$. Denote by $\T$ a finite set of representatives modulo conjugation of all maximal closed subgroups $H$ that admit a conjugate $H_0$ such that $H_0\cap\Gamma$ is infinite. We may require that for each $H_0$ in $\T$, the intersection $\Gamma\cap H_0$ is infinite. For each such $H_0$, we will construct a vector space $\mathcal{S}$ over a local field $K$ and a representation of $\Gamma$ in $\mathcal{S}$ such that:
\begin{itemize}
\item the group $\Gamma$ acts proximally and strongly irreducibly on $\mathcal{S}$,
\item if $H$ is any conjugate of $H_0$, then $H\cap\Gamma$ stabilizes a proper subspace of $\mathcal{S}$.
\end{itemize}
As $\Gamma\cap H_0$ is infinite, it contains a non-torsion element $\gamma$. Then, $\Ad\gamma$ has an eigenvalue $\lambda$ that is not a root of unity.
If $\sk$ is the field generated by the coefficients of all $\Ad g$, $g\in\Gamma$, by \cite[Lemma~4.1]{titsalternative}, we may choose an embedding of $\sk(\lambda)$ into a local field $K_v$ such that $|\lambda|_v>1$.\\
Denote by $\Delta$ the root system of $G$ and by $E$ the Euclidean space containing it. For some $X_0\in E$, the eigenvalues of $\Ad\gamma$ are: $1$ (with multiplicity $\rk G$) and the $e^{i\langle\alpha,X_0\rangle}$, $\alpha\in\Delta$.\\
As $|\cdot|_v$ is multiplicative, there exists a unique $X\in E$ such that
$$\forall \alpha\in\Delta,\quad \log|e^{i\langle\alpha,X_0\rangle}|_v = \langle\alpha,X\rangle.$$
We choose a basis for $\Delta$ such that $X$ lies in the Weyl chamber $C$ and consider the associated complex irreducible representation of $G$ introduced earlier as $\mathcal{S}_X$.
We choose a finite extension $K$ of $K_v$ containing all extensions of $\sk$ of degree at most $\dim\mathcal{S}_X$ and such that $\bG$ is split over $K$. The representation $\mathcal{S}_X$ is then defined over $K$, and we set $\mathcal{S}=\mathcal{S}_X(K)$. As $\Gamma$ is a Zariski dense subgroup of $\bG(K)$, $\mathcal{S}$ is a strongly irreducible and proximal representation of $\Gamma$.\\
On the other hand, writing the largest root $\tilde{\alpha}=\omega$ or $\tilde{\alpha}=\omega+\omega^*$, Assumption~(\ref{notomega}) implies that the element $X$ is not collinear to $\omega$ nor to $\omega^*$. Moreover, the vector $\xi_X$ is the eigenvector of $\gamma$ associated to the unique eigenvalue of maximal modulus in $K_v$, so that Lemma~\ref{nonirreducible} shows that $\mathcal{S}_X$ is not irreducible under $H_0$.
As we already observed, this implies that whenever $H$ is conjugate to $H_0$, $\mathcal{S}_X$ is not irreducible under $H$.\\
Thus, if $H$ is any conjugate of $H_0$, applying Lemma~\ref{boundedextension} below to the set of $\Ad g$, for $g\in\Gamma\cap H$, we obtain an extension $K'>K$ of degree at most $\dim\mathcal{S}_X$ and a proper subspace of $\mathcal{S}_X$ defined over $K'$ that is stable under $\Gamma\cap H$. This yields a proper subspace of $\mathcal{S}$ stable under $\Gamma\cap H$ and finishes the proof.
\end{proof}

For convenience of the reader, we recall the following easy linear algebra lemma, which we just used in the above proof.

\begin{lemma}\label{boundedextension}
Let $A$ be a subset of $SU(d)$ whose elements have coefficients in a field $\sk<\mathbb{C}$, and suppose $A$ stabilizes a proper subspace $V$ of $\mathbb{C}^d$. Then there exists an extension $\sk'>\sk$ of degree at most $d$ and a proper subspace $V'$ defined over $\sk'$ and stable under $A$.
\end{lemma}
\begin{proof}
The set of solutions $x\in\End(\mathbb{C}^d)$ to
\begin{equation}\label{linear}
\forall a\in A,\ ax = xa,
\end{equation}
is a vector space defined over $\sk$, it contains both the identity and the orthogonal projection on the proper stable subspace, so it has dimension at least two. Therefore, we may find a solution $x$ that has coefficients in $\sk$ and is not a homethety. Then, pick an eigenvalue $\lambda$ of $x$, let $\sk'=\sk(\lambda)$ and $V'=\ker(x-\lambda I)$; this solves the problem.
\end{proof}

\subsection{From a closed subgroup to a small neighborhood}

Let $S$ be a finite set of \emph{algebraic} elements in $G$, and let $\Gamma=\langle S\rangle$ be the subgroup generated by $S$.
We endow $\Gamma$ with the word metric associated to the generating system $S$, and denote by $B_\Gamma(n)$ the ball of radius $n$ centered at the identity, for that metric.
If $L$ is a proper subspace of the Lie algebra $\mathfrak{g}$ of $G$, we let
$$H_L = \{g\in G \,|\, (\Ad g)L=L \}.$$
The key proposition is the following.

\begin{proposition}\label{neighborhood}
Let $G$ be a connected compact simple group and $\Gamma$ a dense subgroup generated by a finite set $S$ of algebraic elements of $G$.
There exist a constant $C_1=C_1(S)$ and an integer $n_0$ such that for any integer $n\geq n_0$, for any proper subspace $L_0<\mathfrak{g}$, there exists a proper closed subgroup $H_1<G$ such that
$$B_\Gamma(n)\cap H_{L_0}^{(e^{-C_1n})} \subset B_\Gamma(n)\cap H_1.$$
\end{proposition}

With this proposition, let us prove Theorem~\ref{algebraicspectralgap}.

\begin{proof}[Proof of Theorem~\ref{algebraicspectralgap}]
By Theorem~\ref{adsg}, it suffices to check that $\mu$ is almost Diophantine. Let $C_1$ be the constant given by Proposition~\ref{neighborhood}.
For $H$ a proper closed subgroup of $G$ we want to bound $\mu^{*n}(H^{(e^{-C_1n})})$. If $H$ is finite we conclude as in the proof of Lemma~\ref{goodrepresentation} using Selberg's Lemma and Kesten's Theorem, so we may as well assume that $H$ is positive dimensional. Denote by $L_0$ its Lie algebra. By Proposition~\ref{neighborhood},
$$B_\Gamma(n)\cap H_{L_0}^{(e^{-C_1n})} \subset B_\Gamma(n)\cap H_1,$$
and therefore, by Proposition~\ref{subgroups} (taking $c_2=\kappa>0$),
$$\mu^{*n}(H^{(e^{-C_1 n})}) \leq \mu^{*n}(H_1) \leq e^{-c_2 n},$$
and $\mu$ is almost Diophantine.
\end{proof}

To prove Proposition~\ref{neighborhood} we want to use an effective version of Hilbert's Nullstellensatz. For that, we need to set up some notation.

\bigskip

Let $e_i$, $1\leq i\leq d$, be a basis for $\mathfrak{g}_{\mathbb{C}}$, and define, for $I\subset\{1,\dots,d\}$,
$$e_I=\bigwedge_{i\in I}e_i.$$
The family $(e_I)_{|I|=l}$ is a basis for $\bigwedge^\ell\mathfrak{g}_{\mathbb{C}}$.
Denote $\mathcal{W}_\ell\subset\bigwedge^\ell\mathfrak{g}_{\mathbb{C}}$ the set of pure tensors, i.e. the set of elements in $\bigwedge^\ell\mathfrak{g}_{\mathbb{C}}$ that can be written $v_1\wedge v_2\wedge\dots\wedge v_\ell$ for some $v_i$'s in $\mathfrak{g}_{\mathbb{C}}$. It is easy to check that $\mathcal{W}_\ell$ is an algebraic subvariety of $\bigwedge^\ell\mathfrak{g}_{\mathbb{C}}$ defined over the rationals and therefore, we may choose a finite collection of polynomials $(R_j)_{1\leq j\leq C}$ with integer coefficients in ${d\choose \ell}$ variables such that for any $v=\sum v_I e_I$ in $\bigwedge^\ell \mathfrak{g}_{\mathbb{C}}$,
$$v\in \mathcal{W}_\ell \Longleftrightarrow \forall j,\ R_j((v_I)_{|I|=\ell})=0.$$
We also define a family of polynomial maps $P_{I_0,g}:\mathbb{C}^{{d\choose \ell}-1}\rightarrow \bigwedge^\ell\mathfrak{g}_{\mathbb{C}}$ for $I_0\subset\{1,\dots,d\}$ with $|I_0|=\ell$ and $g\in G$, in the following way. The polynomial $P_{I_0,g}$ has ${d\choose \ell}-1$ variables $v_I$, indexed by all subsets $I$ of $\{1,\dots,d\}$ of cardinality $\ell$ except $I_0$, and is defined by
$$P_{I_0,g}((v_I)) =  g\cdot v -v,$$
where $v=e_{I_0}+\sum_{I\neq I_0}v_Ie_I$.

\begin{definition}
If $P$ is a polynomial map $\mathbb{C}^a\rightarrow\mathbb{C}^b$ with coefficients in a number field $\sk$ (in the canonical bases), we define the \emph{size} of $P$ by
$$\|P\| = \max\{ |\sigma(c)| \,;\, c\ \mbox{coefficient of}\ P,\ \sigma\in\Hom_{\mathbb{Q}}(\sk,\mathbb{C})\}.$$
\end{definition}

Let $\sk$ be the number field generated by the coefficients of all $\Ad g$, for $g\in\Gamma$, and denote by $\mathcal{O}_{\sk}$ its ring of integers.
We have the following obvious lemma.

\begin{lemma}
There exists a positive integer $q=q(S)$ such that if $g\in B_\Gamma(n)$, then $q^nP_{I_0,g}$ has coefficients in $\mathcal{O}_{\sk}$ and
$$\|q^n P_{I_0,g}\| \leq q^{2n}.$$
\end{lemma}

We are now ready to derive Proposition~\ref{neighborhood}. The letter $C$ denotes any constant that depends only on $G$; this constant will change along the proof.

\begin{proof}[Proof of Proposition~\ref{neighborhood}]
Let $L_0$ be an $\ell$-dimensional subspace of $\mathfrak{g}$ with orthonormal basis $(u_i)_{1\leq i\leq \ell}$. Write $u=u_1\wedge\dots\wedge u_\ell=\sum_{I} u_Ie_I$. As $L_0$ is defined over the reals, $H_{L_0}\cdot u=\pm u$. We assume for simplicity that $H_{L_0}\cdot u = u$.
\footnote{Otherwise, one should use polynomials $P_{I_0,g}(v)$ defining the subvariety $\{v \,|\, g\cdot v\pm v=0\}$.}
For some $I_0$, we have $|u_{I_0}|\geq \frac{1}{C}$ for some constant $C$ depending only on $\dim G$. We let $u'=\frac{1}{|u_{I_0}|}u$, so that $\|u'\| \leq C$.
We claim that if we choose $C_1$ large enough, then, for $n\geq n_0$ ($C_1,n_0$ independent of $L_0$), the family of polynomials $\mathcal{P}=\{R_i\}\cup\{P_{I_0,g}\}_{g\in H_{L_0}^{(e^{-C_1n})}\cap B_\Gamma(n)}$ must have a common zero in $\mathbb{C}^{{d\choose \ell}-1}$.\\
Suppose for a contradiction that this is not the case.
By the above lemma, there is a positive integer $q$ depending only on $S$ such that for all $P$ in $\mathcal{P}$, $q^nP$ has coefficients in $\mathcal{O}_{\sk}$ and for all $P$ in $\mathcal{P}$,
$$\|q^nP\| \leq q^{2n}.$$
As the $P_{I_0,g}$ have bounded degree (in fact, degree $1$) we may extract from the family $q^n\mathcal{P}$ polynomials $P_j$, $1\leq j\leq C$ generating the same ideal as $\mathcal{P}$.
By the effective Nullstellensatz \cite[Theorem~IV]{masserwustholz}, if the family of polynomials $\mathcal{P}$ has no common zero, then there exist an element $a\in\mathcal{O}_{\sk}$ and polynomials $Q_j$ with coefficients in $\mathcal{O}_{\sk}$, such that
\begin{equation}\label{bezout}
a = \sum Q_jP_j
\end{equation}
and
\begin{equation}\label{effective}
\forall j,\quad \|Q_j\| \leq q^{Cn} \quad \deg Q_j\leq C \quad\mbox{and}\quad \|a\|\leq q^{Cn}.
\end{equation}
Now, we want to evaluate (\ref{bezout}) at $u'$ to get a contradiction.\\
First, we observe that for any $P$ in $q^n\mathcal{P}$ (in particular, for any $P_j$),
$$|P(u')|\leq Cq^n e^{-C_1n}.$$
Indeed, if $P$ is one of the $R_i$'s, we have $P(u')=0$ because $u'$ is a pure tensor; and if $P=P_{I_0,g}$, using that $g\in H_{L_0}^{(e^{-Cn})}$ and that $H_{L_0}$ fixes $u'$,
we also find the desired estimate.\\
Second, by (\ref{effective}) and the fact that $\|u'\|\leq C$, we have, for each $j$,
$$|Q_j(u')|\leq C q^{Cn}.$$
Finally, as $a$ is a nonzero element of $\mathcal{O}_{\sk}$ of size at most $q^{Cn}$, we have a lower bound on its complex absolute value (for a constant $M$ depending only on $\mathcal{O}_{\sk}$):
$$q^{-Mn} \leq |a|.$$
Thus,
$$q^{-Mn} \leq |a| \leq \sum |Q_j(u')||P_j(u')| \leq C q^{Cn} e^{-C_1n},$$
which yields a contradiction provided we have chosen $C_1$ large enough (in terms of $C$, $q$ and $M$).\\
Now let $(v_I)_{I\neq I_0}$ be a common zero for the family $\mathcal{P}$. As, for each $i$, $R_i((v_I))=0$, the vector $v=e_{I_0}+\sum_{I\neq I_0} v_Ie_I$ is a pure tensor: $v=v_1\wedge\dots\wedge v_\ell$. Moreover, for all $g$ in $B_\Gamma(n)\cap H_{L_0}^{(e^{-Cn})}$, $g\cdot v=v$, so that the subspace $L_1=\Span v_i$ is stable under $g$. In other terms, $g\in H_{L_1}$, which is what we wanted to show.
\end{proof}

\bibliographystyle{plain}
\bibliography{bibliography}

\end{document}